\documentclass[twoside,bezier,reqno]{amsart}

\usepackage{epsfig}
\usepackage{amsmath,amsthm,amsfonts,amssymb,amscd,euscript}

\input{epsf}
\newcommand{\filebegin}{\begin{document}}
\newcommand{\fileend}{\end{document}}
\makeatother
\def\thefootnote{}
\newcommand{\lo}{\longrightarrow}
\newcommand{\NMM}{\hspace*{2mm}}

\renewcommand{\baselinestretch}{1.1}

\input{epsf}
\renewcommand{\baselinestretch}{1.1}
\def\n{\noindent}%

\numberwithin{equation}{section}
\def\mapdown#1{\Big\downarrow\rlap
{$\vcenter{\hbox{$\scriptstyle#1$}}$}}
\newtheorem{theorem}{Theorem}[section]
\newtheorem{lemma}[theorem]{Lemma}
\newtheorem{proposition}[theorem]{Proposition}
\newtheorem{corollary}[theorem]{Corollary}

\theoremstyle{definition}
\newtheorem{definition}[theorem]{Definition}
\newtheorem{example}[theorem]{\sc Example}
\newtheorem{xca}[theorem]{Exercise}
\theoremstyle{remark}
\newtheorem{remark}[theorem]{Remark}

\begin{document}


\vspace*{2cm}
\begin{center}
{\bf\large Torse-forming vector field with certain deformations}
 \\[0.5cm]
{Beldjilali Gherici$^{1*}$\footnote{$^*$Corresponding Author},\: Benaoumeur Bayour$^{2}$ and  Habib Bouzir$^{3}$\\[2mm]
$^{1,2,3}$Laboratory of Quantum Physics of Matter  and Mathematical Modeling (LPQ3M),
University of Mascara, Algeria\\
[2mm]
{\tt E-mail: gherici.beldjilali@univ-mascara.dz,\: b.bayour@univ-mascara.dz,\: habib.bouzir@univ-mascara.dz}
} \\[2mm]
\end{center}%
\vspace*{0.5cm}
\begin{quotation}
\noindent
{\footnotesize
{\sc Abstract.}
Torse-forming vector fields are generalization of some important vector fields. In this paper, we present some techniques to transform a proper torse-forming vector field  into its special cases. Concrete examples are given.

}
\end{quotation}

\n \textbf{2000 Mathematics subject classification: } 53C15, 53C25, 53D10.

\markboth 
{ G. Beldjilali and ...}
 {Torse-forming vector field with certain deformations}


\section{Introduction}

It is well known that generalizing mathematical concepts is a prerequisite for most researchers. It is also often an important means of developing mathematics and answering many questions that have occupied the minds of researchers in many fields. For example, the notion of a Ricci soliton is a natural generalization of an Einstein metric; at the same time, it is a stationary solution of a famous PDE for Riemannian metrics, known under the name of the Ricci flow equation. Without forgetting that Ricci soliton itself has had many generalizations made recently in many works.

Generally, a mathematical concept is initially defined by a simple expression that satisfies the need to include it. Then questions begin: What if we replace the constant with a function? What if we add a differential one-form? What happens if we add a certain term? ... And so, little by little, the simple expression begins to take on longer and more complex forms. 

In principle, if these generalizations fail to answer any outstanding mathematical problems or fail to impact other concepts, then they can be considered an intellectual luxury. 

In this paper we present a step in the reverse direction of generalization; as a subject of study, we will discuss the concept of the torse-forming vector field as a generalization of both the  torqued vector field, the concircular vector field, the concurrent vector field, the recurrent vector field, and the parallel vector field.

 Torse-forming vector fields  demand special attention due to their applications not only in relativity and cosmology but in theory of submanifolds also. On a Riemannian manifold $(M^n, g)$, a the torse-forming vector field is a vector field $V$ that satisfies
\begin{equation}\label{TFVF}
		\nabla_X V = fX  + \theta(X)V,
	\end{equation}
for any vector field $X$ on $M$ where $\nabla$ is the Levi-Civita connection of $g$, $\theta$ is a one-form and $f$ is a smooth function on $(M^n, g)$. The 1-form $\theta$ is called the generating form and the function $f$ is called the conformal scalar.  If $\omega$ is the corresponding  $1$-form of $V$ , i.e., $\omega(X) = g(X,V)$ then for all $X, Y$ vector fields on $M$
\begin{equation}\label{TFVF2}
 (\nabla_X \omega)Y = g(\nabla_X V , Y)= fg(X,Y)+ \theta(X)\omega(Y).
	\end{equation}
 Some special types of torse-forming vector fields have been considered in various studies. A torse-forming vector field $V$ is called: 
\begin{itemize}
\item[1)] proper torse-forming if $f \neq 0$ and the 1-form $\theta$ is nowhere zero on a dense open subset of $M$.
\item[2)] torqued vector field if $V$ satisfying (\ref{TFVF}) with $\theta(V)=0$ (see \cite{CBY2, CBY3}).
\item[3)] concircular vector field  if $\theta$ is identically zero \cite{YK2} .
\item[4)] concurrent vector field if $\theta = 0$ and $f=1$.
\item[5)] reccurent vector field if $\theta \neq 0$ and $f=0$.
\item[6)] geodesic vector field if $\nabla_V V=0$.
\item[7)] parallel vector field if $\theta = 0$ and $f=0$.
\end{itemize}

\textbf{Agreement:} Through the rest of this paper, $(M,g)$ always denotes a Riemannian manifold.  $V$ denotes a proper torse-forming vector field on $M$ that satisfies (\ref{TFVF}) with $\Vert V \Vert = {\rm e}^{\rho}$, and  $\omega $  the  $1$-form corresponding to $V$, i.e.,  for all $X$ vector field on $M$,
$$ \omega(X)=g(V , X)\qquad and \qquad  \omega(V)={\rm e}^{2\rho}.$$

\begin{proposition}\label{Prop1}
For every  torse-forming vector field $V$ that satisfies (\ref{TFVF}),  we have
\begin{equation}\label{Theta}
		\theta ={\rm d}\rho - f {\rm e}^{-2 \rho} \omega.
	\end{equation}
\end{proposition}
\begin{proof}
Using (\ref{TFVF}), for all $X$ vector field on $M$, we have
\begin{eqnarray}\label{Eq100}
g(\nabla_X V, V) &=& f g(X,V)  + \theta(X)g(V,V) \notag\\
&=& f \omega(X) + {\rm e}^{2\rho} \theta(X).
\end{eqnarray}
On the other hand, we have
\begin{eqnarray}\label{Eq110}
g(\nabla_X V, V)= X g(V,V)-g(V , \nabla_X V)  &\Leftrightarrow &   g(\nabla_X V, V) = \frac{1}{2} X({\rm e}^{2\rho}) \notag\\
&\Leftrightarrow & g(\nabla_X V, V) = {\rm e}^{2\rho} X(\rho).
\end{eqnarray}
Then, from (\ref{Eq100}) and (\ref{Eq110}) we get
$$ \theta(X) =  X(\rho) - f {\rm e}^{-2\rho} \omega(X).$$
This is what is required to be proven,

\end{proof}
\section{Torse-formming vector fields and deformations of  metrics}
\subsection{Conformal transformation}
Let $V$ be a proper torse-formming vector field on $(M,g)$ which satisfying (\ref{TFVF}). We know that a conformal transformation involves changing a Riemannian metric by multiplying it with a positive scalar function, known as the conformal factor. This change preserves the angles between vectors at each point, making the two metrics "conformally equivalent". The new metric $\tilde{g}$ is related to the original metric $g$ by the equation 
$$\tilde{g}={\rm e}^{2\sigma}g,$$
 where $\sigma$ is the conformal factor.

For a conformal transformation, it is well known that the Levi-Civita connection $\nabla$ and $\tilde{\nabla}$ associated with the metrics $g$ and $\tilde{g}$ respectively are connected by
\begin{eqnarray}\label{TransfConf}
\tilde{\nabla}_{X}Y= \nabla_{X}Y + X(\sigma)Y + Y(\sigma)X - g(X,Y){\rm grad}\sigma.
	\end{eqnarray}
By using (\ref{TFVF}) and Proposition \ref{Prop1}, we have
\begin{eqnarray}\label{TildNabV}
\tilde{\nabla}_{X}V &=& \nabla_{X}V + X(\sigma)V + V(\sigma)X - g(X,V){\rm grad}\sigma \notag\\
&=&  f X + \big( X(\rho) - f {\rm e}^{-2\rho} \omega(X) \big)V + X(\sigma)V + V(\sigma)X - \omega(X){\rm grad}\sigma \notag\\
&=&  \big(f + V(\sigma) \big) X + \big( X(\rho) + X(\sigma) \big)V - \omega(X) \big( f {\rm e}^{-2\rho} V +{\rm grad}\sigma \big).
\end{eqnarray}
It is clear that $V$ is a  torse-formming vector field on $(M,\tilde{g})$ if and only if
$${\rm grad}\sigma = -f {\rm e}^{-2\rho} V.$$
Also, under this condition, we get $ f + V(\sigma)=0$ and (\ref{TildNabV}) gives
$$\tilde{\nabla}_{X}V = \theta(X)V.$$
Consequently, $V$ is recurrent with respect to $\tilde{g}$. Hence, we conclude the following:
\begin{theorem}\label{Th1}
Let $V$ be a proper torse-formming vector field on $(M,g)$ which satisfying (\ref{TFVF}). $V$ is recurrent with respect to $\tilde{g}={\rm e}^{2\sigma}g$ where $\sigma \in \mathcal{C}^{\infty}(M)$ if and only if ${\rm grad}\sigma = -f {\rm e}^{-2\rho} V$.
\end{theorem}
\begin{example}
It is convenient to have some sort of hyperspherical coordinates on the unit sphere $\mathbb{S}^3$ in analogy to the usual spherical coordinates on $\mathbb{S}^2$. One such choice is to use $(x_1, x_2,x_3)$ where
\begin{equation}\label{parametrisationS3}
 \left\{
          \begin{array}{lll}
x_1 = \cos x_1,\\
x_2 = \sin x_1 \cos x_2,\\
x_3 = \sin x_1 \sin x_2 \cos x_3,\\
x_4 = \sin x_1 \sin x_2 \sin x_3,
          \end{array}
  \right.
\end{equation}
where $x_1$ and $x_2$ run over the range $0$ to $\pi$, and $x_3$ runs over $0$ to $2 \pi$.
  
One can define the Riemannian metric $g$ in these coordinates as follows:
$$ g = d x_1^2 + \sin^2 x_1 ( dx_2^2 + \sin^2 x_2 \;d x_3^2).$$
where
$$\left\lbrace   e_1 = \frac{\partial}{\partial x_1},\quad e_2 = \frac{1}{\sin x_1}\frac{\partial}{\partial x_2},\quad
e_3=\frac{1}{\sin x_1 \sin x_2}\frac{\partial}{\partial x_3}\right\rbrace ,$$
is an orthonormal basis for $g$. By Kozsul's formula, the covariant derivatives of the basis elements are as follows:
$$\begin{tabular}{lll}
$ \nabla_{e_1} e_1 = 0,$ & $\nabla_{e_1} e_2 =0,$ & $\nabla_{e_1} e_3 = 0$, \\ 
$\nabla_{e_2} e_1= \cot x_1 e_2,$ &  $\nabla_{e_2} e_2= -\cot x_1 e_1,$ & $\nabla_{e_2} e_3 =0,$ \\ 
$\nabla_{e_3} e_1= \cot x_1 e_3,$&  $\nabla_{e_3} e_2 = \frac{\cot x_2 }{\sin x_1} e_3,$ & $\nabla_{e_3} e_3 = -\cot x_1 e_1 - \frac{\cot x_2 }{\sin x_1} e_2.$ \\ 
\end{tabular} 
$$
Easily, we can see that for $ V = h(x_3) \sin x_1 e_1$ we have 
\begin{equation}\label{NabV1}
 \left\{
          \begin{array}{lll}
\nabla_{e_1} V =h(x_3) \cos x_1 e_1 ,\\
\nabla_{e_2} V = h(x_3) \cos x_1 e_2,\\
\nabla_{e_3} V =\frac{h'}{\sin x_2 } e_1 + h(x_3) \cos x_1 e_3,
          \end{array}
  \right.
\end{equation}
where $h' = \frac{\partial h}{\partial x_3}$. Then, for $\theta = \frac{h'}{h} dx_3$ we get
$$\nabla_{e_i} V =  h(x_3) \cos x_1 e_i + \theta(e_i) V,$$
hence, according to (\ref{TFVF}), $V$ is a proper torse-forming vector field on $(\mathbb{S}^3, g)$ with $f =h(x_3) \cos x_1$ and $\theta = \frac{h'}{h} dx_3$.

لBy taking $\tilde{g}= {\rm e}^{2 \sigma}g$ the condition ${\rm grad}\sigma = -f {\rm e}^{-2\rho} V$ gives the system
$$ \sigma_1 =- \cot x_1,\quad \sigma_2 =0 \quad and \quad \sigma_3 =0,$$
implies $\sigma = \ln \frac{c}{\sin x_1}$ where $c \in \mathbb{R}$. Hence,
$$ \tilde{g} = c^2\left( \frac{dx_1^2}{\sin^2 x_1}  + dx_2^2 + \sin^2 x_2 dx_3^2\right),$$
and the non zero components of the Levi-Civita connection corresponding to $ \tilde{g}$ are given by:
  $$ \tilde{\nabla}_{\tilde{e}_3} \tilde{e}_2 = \frac{\cot x_2}{c} \tilde{e}_3,  \quad and \quad \quad \tilde{\nabla}_{\tilde{e}_3} \tilde{e}_3 = -\frac{\cot x_2}{c} \tilde{e}_2,$$
where $\tilde{e}_i = \frac{\sin x_1}{c} e_i$  is an orthonormal basis for $\tilde{g}$.  

Now, with simple calculations we can verify that 
$$  \tilde{\nabla}_{\tilde{e}_i} V= \theta(\tilde{e}_i)V,$$ 
 which means that $V$ is a  recurrent vector field on $(\mathbb{S}^3, \tilde{g})$, which confirms the result of  theorem \ref{Th1}.

\end{example} 
\subsection{$\mathcal{D}$-isometric defomation}

Knowing that $\omega$ is a global $1$-form on $(M^n,g)$, the equation $\omega = 0$ defines an $(n-1)$-dimensional distribution $\mathcal{D}$ on $M$. We define on $M$ a Riemannian metric, denoted $\overline{g}$, by
\begin{equation}\label{TildeMetric}
\overline{g}(X, Y )=  g(X,Y)+ \omega(X)\omega(Y).
\end{equation}
Then, we have
$$\left\lbrace 
       \begin{array}{ll}
          \overline{g}(V, V)={\rm e}^{2\rho}+{\rm e}^{4\rho}\\
          \overline{g}(X , X)=g(X,X)\qquad \forall X \in \mathcal{D}
       \end{array}
\right.
$$
That is why, we refer to this construction as $\mathcal{D}$-isometric deformation.

\begin{proposition}
 Let $\nabla$ and $\overline{\nabla} $  denote the Riemannian connections of $ g$ and $\overline{g}$ respectively. Then, for all $X,Y $ vector fields on $M$, we have the relation:
\begin{eqnarray}\label{NablaOverline}
\overline{\nabla}_{X} Y &=& \nabla_X Y - \omega(X)\omega(Y){\rm grad}\rho\notag\\
&+& \frac{1}{1+{\rm e}^{2\rho}} \Big( f g(X,Y) +X(\rho) \omega(Y) + \omega(X)Y(\rho)\notag\\
&&\qquad\qquad +\big(V(\rho)-f {\rm e}^{-2\rho}\big)\omega(X)\omega(Y)\Big)V.
 \end{eqnarray}
\end{proposition}
\begin{proof}
Using the Koszul formula for the Levi-Civita connection of a Riemannian metric $\overline{g}$, 
\begin{eqnarray*}
		2\overline{g}(\overline{\nabla}_X Y,Z) &=& X \overline{g}(Y, Z) + Y \overline{g}(Z, X) - Z \overline{g}(X, Y )\\
		&& -\overline{g}(X ,[Y, Z]) + \overline{g}(Y,[Z,X ]) + \overline{g}(Z,[ X,Y]),
	\end{eqnarray*}
one can obtain 
\begin{align*}
2\overline{g}(\overline{\nabla}_{X} Y,Z) = 2\overline{g}(\nabla_{X} Y,Z) &+ \big( (\nabla_X \omega)Y + (\nabla_Y \omega)X \big) \omega(Z)  \notag\\
&+ \big( (\nabla_X \omega)Z - (\nabla_Z \omega)X \big) \omega(Y) \notag\\
 &+ \big( (\nabla_Y \omega)Z - (\nabla_Z \omega)Y \big) \omega(X).
 \end{align*}
 Better yet, the following relation:
\begin{align}\label{Eq200}
2\overline{g}(\overline{\nabla}_{X} Y,Z) = 2\overline{g}(\nabla_{X} Y,Z) &+ ( \mathcal{L}_Vg)(X,Y) \omega(Z)  \notag\\
&+ 2 {\rm d}\omega(X,Z) \omega(Y)+ 2 {\rm d}\omega(Y,Z) \omega(X), 
 \end{align}
where $\mathcal{L}_V g$ is the Lie derivative of $g$ along $V$ and ${\rm d}$ denotes the exterior derivative.
Using (\ref{TFVF}), one can get
  \begin{eqnarray}\label{Eq205}
\omega(Z) &=& g(V,Z) \notag\\
&=& \overline{g}(\xi,Z) - {\rm e}^{2\rho}\omega(Z),
 \end{eqnarray}
 then,
 \begin{equation}\label{Eq210}
 \omega(Z) = \frac{1}{1+ {\rm e}^{2\rho}}\overline{g}(V,Z).
 \end{equation}
Also, we have
  \begin{eqnarray}\label{Eq215}
Z(\rho) &=& g({\rm grad} \rho,Z) \notag\\
&=& \overline{g}({\rm grad} \rho,Z) - \omega({\rm grad} \rho)\omega(Z) \notag\\
&=& \overline{g}\Big({\rm grad} \rho - \frac{V(\rho)}{1+ {\rm e}^{2\rho}}V,Z \Big).
 \end{eqnarray}
 So, by replacing (\ref{Eq210}) and (\ref{Eq215}) in (\ref{Theta}) we obtain
\begin{equation}\label{Eq220}
\theta(Z)= \overline{g}\Big( {\rm grad}\rho - \frac{1}{1+ {\rm e}^{2\rho}}\big( V(\rho) + f {\rm e}^{-2 \rho}\big)V,Z \Big).
\end{equation} 
Now, with the help of (\ref{TFVF}), (\ref{Eq210}) and (\ref{Eq220}) we have
  \begin{eqnarray}\label{Eq215}
 2 {\rm d}\omega(X,Z)  &=&g(\nabla_X V,Z) - g(\nabla_Z V,X) \notag\\
&=&  \theta(X)\omega(Z) - \theta(Z)\omega(X) \notag\\
&=& \overline{g}\Big( - \omega(X){\rm grad}\rho + \frac{1}{1+{\rm e}^{2\rho}} \big( X(\rho) + V(\rho)\omega(X) \big) V,Z \Big).
 \end{eqnarray}
Therefore, by substituting (\ref{Eq210}) and (\ref{Eq215}) in (\ref{Eq200}), we get
\begin{eqnarray}\label{Eq220}
2\overline{g}(\overline{\nabla}_{X} Y,Z) &=& 2\overline{g}(\nabla_{X} Y,Z)-
2 \omega(X)\omega(Y) \overline{g}({\rm grad}\rho, Z)\notag\\
&&+ \frac{1}{1+{\rm e}^{2\rho}}\Big((\mathcal{L}_Vg)(X,Y) +X(\rho) \omega(Y) + Y(\rho)\omega(X) \notag\\
&&\qquad\qquad\qquad + 2V(\rho)\omega(X)\omega(Y)\Big)\overline{g}(V,Z).  
 \end{eqnarray}
Also, using (\ref{Theta}) we compute
  \begin{eqnarray*}\label{Eq225}
 ( \mathcal{L}_Vg)(X,Y)  &=&g(\nabla_X V,Z) + g(\nabla_Z V,X) \notag\\
&=& 2 f g(X,Y) +\theta(X)\omega(Y) + \theta(Y)\omega(X) \notag\\
&=& 2 f g(X,Y) +X(\rho) \omega(Y) + Y(\rho)\omega(X) -2 f {\rm e}^{-2\rho} \omega(X)\omega(Y).
 \end{eqnarray*}
Hence, (\ref{Eq220}) becomes
\begin{eqnarray}\label{Eq230}
\overline{g}(\overline{\nabla}_{X} Y,Z) &=& \overline{g}(\nabla_X Y,Z)  - \omega(X)\omega(Y) \overline{g}({\rm grad}\rho, Z)\notag\\
&+& \frac{1}{1+{\rm e}^{2\rho}} \Big( f g(X,Y) +X(\rho)\omega(Y) + Y(\rho)\omega(X)\notag\\
&&\qquad\qquad +\big(V(\rho)-f {\rm e}^{-2\rho}\big)\omega(X)\omega(Y)\Big)\overline{g}(V,Z )  
 \end{eqnarray}
Consequently,
  \begin{eqnarray}\label{Eq230}
\overline{\nabla}_{X} Y &=& \nabla_X Y - \omega(X)\omega(Y){\rm grad}\rho\notag\\
&+& \frac{1}{1+{\rm e}^{2\rho}} \Big( f g(X,Y) +X(\rho) \omega(Y) + Y(\rho)\omega(X)\notag\\
&&\qquad\qquad +\big(V(\rho)-f {\rm e}^{-2\rho}\big)\omega(X)\omega(Y)\Big)V.
 \end{eqnarray}
 This completes the proof.
 \end{proof} 
Now by putting $Y=V$ in (\ref{NablaOverline}) with the use of using (\ref{TFVF}) and (\ref{Theta}) we get
  \begin{eqnarray}\label{Eq240}
\overline{\nabla}_{X} V &=& fX+ \big( \frac{{\rm e}^{2\rho}}{1+{\rm e}^{2\rho}} X(\rho) - f {\rm e}^{-2\rho}\omega(X)\big) V \notag\\
&+& \omega(X) \big( V(\rho)V  -  {\rm e}^{2\rho} {\rm grad}\rho \big).
 \end{eqnarray}
Note that for $V$ to be a torse-forming vector field on $(M,\overline{g})$, it is   necessary  that 
\begin{equation}\label{Eq245}
 V(\rho)V  -  {\rm e}^{2\rho} {\rm grad}\rho  =0.
\end{equation}
This condition gives
\begin{equation}\label{Eq250}
X(\rho) = {\rm e}^{-2\rho}V(\rho)\omega(X),
\end{equation}
by substituting (\ref{Eq245}) and (\ref{Eq250}) in (\ref{Eq240}) we obtain 
  \begin{eqnarray}\label{Eq255}
\overline{\nabla}_{X} V &=& fX+ {\rm e}^{-2\rho}\Big( \frac{1+ 2{\rm e}^{2\rho}}{1+{\rm e}^{2\rho}} V(\rho) - f \Big) \omega(X)V.
 \end{eqnarray}
Based on these facts, we state the following:
\begin{theorem}
Let $V$ be a proper torse-forming vector field on $(M,g)$ with $V^{\flat}=\omega$ and $\omega(V)={\rm e}^{2\rho}$. If
\begin{equation}\label{Eq260}
 V(\rho)V  -  {\rm e}^{2\rho} {\rm grad}\rho  =0,
\end{equation}
then $V$ is a torse-forming with respect to $\overline{g}$. Moreover, if
 $f=\frac{1+ 2{\rm e}^{2\rho}}{1+{\rm e}^{2\rho}} V(\rho)$ then $V$ is a concircular vector field on $(M, \overline{g})$ where $ \overline{g} = g + \omega \otimes \omega$.

In addition, if $V(\rho)=0$ then $V$ is parallel on $(M,\overline{g})$.
\end{theorem}

\begin{example}\label{EX1}
Let us denote $(x,y,z)$ the Cartesian coordinates in $M=\mathbb{R}^{3}$. Let $(e_{i})$ be the set of vector fields on $M$ defined by:
$$ e_1= \frac{1}{\lambda(x)}\frac{\partial}{\partial x} ,\qquad
 e_2= \frac{1}{h(x)\alpha(y,z)}\frac{\partial}{\partial y},\qquad
   e_3= \frac{1}{h(x)\beta(y,z)}\frac{\partial}{\partial z},$$
where $\lambda, \alpha, \beta$ and $h$ are three functions on $M$ non zero every where  such that
$$ g(e_i  , e_j )=\delta_{ij},\qquad  \forall i,j \in \{1,2,3\}.$$
 That is, the form of the metric $g$ becomes:
\begin{align*}
g= \left(
       \begin{array}{ccc}
        \lambda^2 & 0 & 0\\
        0 & h^2 \alpha^2 & 0\\
        0 & 0 & h^2 \beta^2
       \end{array}
\right).         
\end{align*}
  
By Kozsul's formula, the components of the Levi-Civita connection corresponding to $g$ are given by

$$\begin{tabular}{lll}
$ \nabla_{e_1} e_1 = 0,$ & $\nabla_{e_1} e_2 =0,$ & $\nabla_{e_1} e_3 = 0$, \\ 
$\nabla_{e_2} e_1=\frac{h'}{\lambda h} e_2,$ &  $\nabla_{e_2} e_2=- \frac{h'}{\lambda h} e_1 - \frac{\alpha_3}{\alpha \beta h} e_3,$ & $\nabla_{e_2} e_3= \frac{\alpha_3}{\alpha \beta h} e_2,$\\ 
$\nabla_{e_3} e_1=\frac{h'}{\lambda h} e_3,$&  $\nabla_{e_3} e_2 = \frac{\beta_2}{\alpha \beta h} e_3,$ & $\nabla_{e_3} e_3 = - \frac{h'}{\lambda h} e_1 -  \frac{\beta_2}{\alpha \beta h} e_2,$ \\ 
\end{tabular} 
$$
  where $ h' = \frac{\partial '}{\partial x_1}$, $ \alpha_i = \frac{\partial \alpha}{\partial x_i}$ and $ \beta_i = \frac{\partial \beta}{\partial x_i}$ with $i \in \{2,3\}$.
  
  Let's take $V = \mu \lambda e_1$ where $\mu=\mu(x)$ is a non zero function on $M$. One can get
  \begin{equation}\label{NabV2}
 \left\{
          \begin{array}{lll}
\nabla_{e_1} V = \frac{(\mu \lambda)'}{\lambda} e_1 ,\\
\nabla_{e_2} V = \frac{\mu h'}{h} e_2,\\
\nabla_{e_3} V =\frac{ \mu  h' }{h}  e_3.
          \end{array}
  \right.
\end{equation}
where $ (\mu \lambda)' = \frac{\partial (\mu \lambda)}{\partial x} $. 
Therefore, we can easily verify that 
$$\nabla_{e_i} V = \frac{\mu h'}{h} e_i + \frac{ h}{\mu \lambda} \left( \frac{\mu \lambda}{h}\right)' dx(e_i)V. $$
Hence, for $h' \neq 0$ and $ \mu \lambda \neq const$ we obtain $3$-parameters family of  proper torse-forming vector field on $(M,g)$.

Now, we apply the $\mathcal{D}$-isometric defomation as follows, we put $\overline{g}=g + \omega \otimes \omega$ and for we get
$$ \overline{g} = \lambda^2(1+\mu^2 \lambda^2) dx^2 + h^2(\alpha^2 dy^2 + \beta^2 dz^2).$$
This Riemannian metric admits
$$ \overline{e}_1= \frac{1}{\sqrt{1+\mu^2 \lambda^2}} e_1,\quad \overline{e}_2= e_2 \quad and \quad\overline{e}_3= e_3, $$
as an orthonormal basis so, the non zero components of the Levi-Civita connection corresponding to $ \overline{g}$ are given by:
 $$\begin{tabular}{lll}
$\overline{\nabla}_{\overline{e}_2} \overline{e}_1=\frac{h'}{\lambda h \sqrt{1+\mu^2 \lambda^2}} \overline{e}_2,$ &  $\overline{\nabla}_{e_2} \overline{e}_2=- \frac{h'}{\lambda h \sqrt{1+\mu^2 \lambda^2}} \overline{e}_1 - \frac{\alpha_3}{\alpha \beta h} \overline{e}_3,$\\
 $\overline{\nabla}_{\overline{e}_2} \overline{e}_3= \frac{\alpha_3}{\alpha \beta h} \overline{e}_2,$ & 
$\overline{\nabla}_{\overline{e}_3} \overline{e}_1=\frac{h'}{\lambda h \sqrt{1+\mu^2 \lambda^2}} \overline{e}_3,$\\
  $\overline{\nabla}_{\overline{e}_3} \overline{e}_2 = \frac{\beta_2}{\alpha \beta h} \overline{e}_3,$ & $\overline{\nabla}_{\overline{e}_3} \overline{e}_3 = - \frac{h'}{\lambda h \sqrt{1+\mu^2 \lambda^2}} \overline{e}_1 -  \frac{\beta_2}{\alpha \beta h} \overline{e}_2.$ 
\end{tabular} 
$$
One can easily check that the condition (\ref{Eq260}) is hold which implies that $V$ is a torse-forming vector field with respect to $\overline{g}$ and 
\begin{equation}
\overline{\nabla}_{\overline{e}_i} V = \frac{\mu h'}{h} \overline{e}_i + \mu^3 \lambda^2\Big( \frac{1 + 2 \mu^2 \lambda^2}{1+\mu^2 \lambda^2} \rho' -\frac{h'}{h}\Big)V.
\end{equation}
So, if $f=\frac{1+ 2{\rm e}^{2\rho}}{1+{\rm e}^{2\rho}} V(\rho)$ that is $\frac{h'}{h}= \frac{1 + 2 \mu^2 \lambda^2}{1+\mu^2 \lambda^2} \rho'$ then
$$ \overline{\nabla}_{\overline{e}_i} V = \frac{\mu h'}{h} \overline{e}_i,$$
i.e., $V$ is a concircular vector field on $(M, \overline{g})$.

Until we confirm that the condition above is verifiable i.e., $\frac{h'}{h}= \frac{1 + 2 \mu^2 \lambda^2}{1+\mu^2 \lambda^2} \rho'$, it is sufficient to take $\mu = 1$ and $\lambda = {\rm e}^{x}$ we get $ \rho = x$ and $ h =  {\rm e}^{x}\sqrt{1+ {\rm e}^{2x}}$.

\end{example}
\subsection{$\omega$-conformal deformation}

In this section, we will combine the previous two transformations together to obtain a new Riemannian metric that we call "\textit{$\omega$-conformal deformation}" or "\textit{$\mathcal{D}$-conformal deformation}".

Let $(M,g)$ be an $n$-dimensional Riemannian manifold. The equation $\omega=0$ defines an $n-1$-dimensional distribution $\mathcal{D}$ on $M$. By a
$\mathcal{D}$-conformal deformation  we mean a change of metric of the form
$$ \hat{g}= {\rm e}^{2 \sigma} \overline{g} = {\rm e}^{2 \sigma} (g + \omega \otimes \omega),$$
where $\sigma$ is a smooth function on $M$.

\begin{proposition}
 Let $\nabla$ and $\hat{\nabla} $  denote the Riemannian connections of $ g$ and $\hat{g}$ respectively. Then, for all $X,Y $ vector fields on $M$, we have the relation:
\begin{eqnarray}\label{NablaHat}
\hat{\nabla}_{X}Y &=& \nabla_X Y  + X(\sigma)Y + Y(\sigma)X \notag\\
&-& g(X,Y){\rm grad} \sigma - \omega(X)\omega(Y){\rm grad}(\sigma +\rho)\notag\\
&+& \frac{1}{1+{\rm e}^{2\rho}} \Big( \big(f + V(\sigma)\big) g(X,Y) +X(\rho) \omega(Y) + \omega(X)Y(\rho)\notag\\
&&\qquad\qquad +\big(V(\rho)+V(\sigma)-f {\rm e}^{-2\rho}\big)\omega(X)\omega(Y)\Big)V.
 \end{eqnarray}
\end{proposition}
\begin{proof}
Using (\ref{TransfConf}), we get
\begin{eqnarray}\label{TransfDConf}
\hat{\nabla}_{X}Y= \overline{\nabla}_{X}Y + X(\sigma)Y + Y(\sigma)X - \overline{g}(X,Y)\overline{{\rm grad}}\sigma.
	\end{eqnarray}
For all $X$ vector field on $M$, we compute
\begin{eqnarray}\label{Eq300}
\overline{g}(\overline{{\rm grad}}\sigma, X)  = X(\sigma)
 &=& g({\rm grad} \sigma, X) \notag\\
&=& \overline{g}( {\rm grad} \sigma, X) - V(\sigma)\omega(X) \notag\\
&=& \overline{g}\Big( {\rm grad} \sigma -\frac{ V(\sigma)}{1+ {\rm e}^{2\rho}} V,X\Big),
\end{eqnarray}	
then,
\begin{equation}\label{Eq310}
\overline{{\rm grad}}\sigma =  {\rm grad} \sigma -\frac{ V(\sigma)}{1+ {\rm e}^{2\rho}} V.
\end{equation}
Now, by substituting (\ref{NablaOverline}) and (\ref{Eq310}) in (\ref{TransfDConf}) taking into account $\overline{g} = g + \omega \otimes \omega$ we get our formula.
\end{proof}
Let's replace $Y$ by $V$ in (\ref{NablaHat}) and using (\ref{TFVF}), we get
\begin{eqnarray}\label{NablaHat2}
\hat{\nabla}_{X}V &=& \big(f + V(\sigma)\big) X + \theta(X)V  + X(\sigma)V  \notag\\
&-& (1+{\rm e}^{2 \rho})\omega(X){\rm grad} \sigma - {\rm e}^{2 \rho}\omega(X){\rm grad}(\rho)\notag\\
&+& \frac{1}{1+{\rm e}^{2\rho}} \Big(  {\rm e}^{2 \rho}X(\rho)  +(1+{\rm e}^{2 \rho})\big(V(\rho)+V(\sigma)\big)\omega(X)\Big)V.
 \end{eqnarray}
 Reorganize this equation using (\ref{Theta}), we find

\begin{eqnarray}\label{Eq400}
\hat{\nabla}_{X}V &=& \big(f + V(\sigma)\big) X \notag\\
&&+ \Big(  X(\sigma) + \big(V(\rho)+V(\sigma)- f {\rm e}^{-2 \rho}\big)\omega(X) + \frac{1+2{\rm e}^{2 \rho}}{1+{\rm e}^{2\rho}}  X(\rho) \Big)V  \notag\\
&&- (1+{\rm e}^{2 \rho})\omega(X){\rm grad} \sigma - {\rm e}^{2 \rho}\omega(X){\rm grad}(\rho).
 \end{eqnarray}
For $V$ to be a torse forming vector field with respect to $\hat{g}$, it is sufficient for it to be
\begin{equation}\label{Eq410}
 {\rm grad} \sigma= - \frac{{\rm e}^{2 \rho}}{1+{\rm e}^{2\rho}}{\rm grad}\rho.
\end{equation}
From this equation, we get
\begin{equation}\label{Eq420}
X(\sigma) = - \frac{{\rm e}^{2 \rho}}{1+{\rm e}^{2\rho}}X(\rho)\qquad and \qquad V(\sigma) = - \frac{{\rm e}^{2 \rho}}{1+{\rm e}^{2\rho}}V(\rho).
\end{equation}
By (\ref{Eq410}) and (\ref{Eq420}), (\ref{Eq400}) becomes
\begin{equation}\label{430}
\hat{\nabla}_{X}V = \Big(f - \frac{{\rm e}^{2 \rho}}{1+{\rm e}^{2\rho}} \Big) X + \Big(  X(\rho) + \Big( \frac{V(\rho)}{1+{\rm e}^{2\rho}} - f {\rm e}^{-2 \rho}\Big)\omega(X) \Big)V.
 \end{equation}
Hence, $V$ is a torse-forming with respect to $\hat{g}$. Therfore,
we can extract two important cases:\\
(1):\quad For $f = \frac{{\rm e}^{2 \rho}}{1+{\rm e}^{2\rho}}$ we get
$ \hat{\nabla}_{X}V = X(\rho)V$ this means that $V$ is reccurent on $(M, \hat{g})$.\\
(2):\quad For $ X(\rho) = \Big( f {\rm e}^{-2 \rho} -\frac{V(\rho)}{1+{\rm e}^{2\rho}}\Big)\omega(X)$, by taking $X=V$ we get $  f =\frac{1+2{\rm e}^{2\rho}}{1+{\rm e}^{2\rho}}V(\rho)$ then we obtain $ \hat{\nabla}_{X}V = V(\rho)X$.
By combining the arguments of the above discussion, we obtain the following theorem:
\begin{theorem}
Let $V$ be a proper torse-forming vector field on $(M,g)$ with $V^{\flat}=\omega$ and $\omega(V)={\rm e}^{2\rho}$. If
\begin{equation}\label{Eq450}
 {\rm grad} \sigma= - \frac{{\rm e}^{2 \rho}}{1+{\rm e}^{2\rho}}{\rm grad}\rho.
\end{equation}
then $V$ is a torse-forming with respect to $\hat{g}$. Moreover,\\
\textbf{(i)}\quad if $f = \frac{{\rm e}^{2 \rho}}{1+{\rm e}^{2\rho}}$ then
 $V$ is reccurent vector field on $(M, \hat{g})$.\\
\textbf{(ii)}\quad if $ {\rm d}\rho = \Big( f {\rm e}^{-2 \rho} -\frac{V(\rho)}{1+{\rm e}^{2\rho}}\Big)\omega$ then  $V$ is a concircular vector field on $(M, \hat{g})$.

In addition, if $V(\rho)=0$ then $V$ is parallel on $(M,\hat{g})$.
\end{theorem}
 \begin{example}
 Based on the above example \ref{EX1}, we  apply the $\omega$-conformal deformation as follows, we put 
 $$\hat{g}= {\rm e}^{2 \sigma} \overline{g} = {\rm e}^{2 \sigma} (g + \omega \otimes \omega).$$ 
This Riemannian metric admits
$$ \hat{e}_1= \frac{{\rm e}^{- \sigma}}{\sqrt{1+\mu^2 \lambda^2}} e_1,\quad \hat{e}_2= {\rm e}^{- \sigma} e_2 \quad and \quad \hat{e}_3= {\rm e}^{- \sigma}e_3, $$
as an orthonormal basis so, the non zero components of the Levi-Civita connection corresponding to $ \overline{g}$ are given by:
 $$\begin{tabular}{lll}
 $\hat{\nabla}_{\hat{e}_1} \hat{e}_1= -{\rm e}^{- \sigma} \Big(\frac{\sigma_2}{h \alpha} \hat{e}_2 + \frac{\sigma_3}{h \beta} \hat{e}_3\Big),$ &  $\hat{\nabla}_{e_1} \hat{e}_2=\frac{\sigma_2 {\rm e}^{- \sigma}}{h \alpha} \hat{e}_1,$\\
 $\hat{\nabla}_{e_1} \hat{e}_3=\frac{\sigma_3 {\rm e}^{- \sigma}}{h \beta} \hat{e}_1,$ & $\hat{\nabla}_{\hat{e}_2} \hat{e}_1=\frac{ {\rm e}^{- \sigma}(\sigma_1 h + h')}{h \lambda \sqrt{1+\mu^2 \lambda^2}} \hat{e}_2,$\\
  $\hat{\nabla}_{e_2} \hat{e}_2=- {\rm e}^{- \sigma}\Big(\frac{ \sigma_1 h + h'}{h \lambda \sqrt{1+\mu^2 \lambda^2}} \hat{e}_1 + \frac{ \sigma_3 \alpha + \alpha_3}{h \alpha \beta} \hat{e}_3 \Big),$ & $\hat{\nabla}_{\hat{e}_2} \hat{e}_3= \frac{ {\rm e}^{- \sigma}(\sigma_3 \alpha + \alpha_3)}{h \alpha \beta} \hat{e}_2,$\\ 
$\hat{\nabla}_{\hat{e}_3} \hat{e}_1=\frac{ {\rm e}^{- \sigma}(\sigma_1 h + h')}{h \lambda \sqrt{1+\mu^2 \lambda^2}} \hat{e}_3,$ &  $\hat{\nabla}_{\hat{e}_3} \hat{e}_2 = \frac{ {\rm e}^{- \sigma}(\sigma_2 \beta + \beta_2)}{h \alpha \beta} \hat{e}_3,$\\
 $\hat{\nabla}_{\hat{e}_3} \hat{e}_3 = - {\rm e}^{- \sigma}\Big(\frac{\sigma_1 h + h'}{h \lambda \sqrt{1+\mu^2 \lambda^2}} \hat{e}_1 + \frac{ \sigma_2 \beta + \beta_2}{h \alpha \beta} \hat{e}_2 \Big),$ 
\end{tabular} 
$$
where $ \sigma_i = \frac{\partial \sigma}{\partial x_i}$. In the above example \ref{EX1}, it was proven that $V = \mu \lambda e_1$ where $\mu=\mu(x)$ is a proper torse-forming vector field on $(M,g)$ for $h' \neq 0$ and $ \mu \lambda \neq const$. 

From condition (\ref{Eq450}), one can get $\sigma = -\ln\sqrt{1+\mu^2 \lambda^2}$
Then we get 
$$ \hat{\nabla}_{\hat{e}_i}V = \hat{f} X + \hat{\theta} V \qquad where \quad \hat{f} = \mu \Big(\frac{h'}{h}-\frac{u \lambda (\mu \lambda)'}{1 + \mu^2 \lambda^2}\Big)\quad and \quad \hat{\theta} = \frac{(\mu \lambda)'}{ u \lambda }- \hat{f},$$
i.e., $V$ is a torse-forming vector field on $(M, \hat{g}$. 

Now, as we did in the example above, let's take $\mu = 1$ and $\lambda = {\rm e}^{x}$ then, if we put $ \hat{f}=0$ we get $ h = \sqrt{1+{\rm e}^{2 x}}$ and $V$ becomes reccurent on $(M, \hat{g})$. But, if we put $\hat{\theta}=0$ we find $ h = {\rm e}^{x}\sqrt{1+{\rm e}^{2 x}}$ and $V$ becomes concircular on $(M, \hat{g})$.

 \end{example}

\providecommand{\bysame}{\leavevmode\hbox
to3em{\hrulefill}\thinspace}



\begin{thebibliography}{99}
        

\bibitem{CBY1} Chen, B. Y., Some results on concircular vector fields and their applications to Ricci solitons, Bull. Korean Math. Soc., 52(2015),
1535-1547.
\bibitem{CBY2} Chen, B. Y., Rectifying submanifolds of Riemannian manifolds and torqued vector fields, Krajujevac Math. J., 41(2017), 93-103
\bibitem{CBY3} Chen, B. Y., Classification of torqued vector fields and its applications to Ricci solitons. Kragujevac J. Math. 41(2), 239–250 (2017).
\bibitem{YK2} Yano, K., Concircular geometry I, Concircular transformations, Proc. Imp. Acad. Tokyo, 16(1940), 195-200.
\end{thebibliography}
\end{document}